\pdfoutput=1

\providecommand{\noopsort}[1]{} 

\documentclass{amsart}

\usepackage[mathscr]{eucal}
\usepackage{amssymb}
\usepackage[usenames,dvipsnames]{xcolor} 
\usepackage[normalem]{ulem}
\usepackage{amsthm}
\usepackage{bbold}
\usepackage{comment}
\usepackage{enumitem}
\usepackage{amsmath}
\usepackage{tikz-cd}

\usepackage{etoolbox} 

\usepackage[unicode]{hyperref} 

\definecolor{dark-red}{rgb}{0.5,0.15,0.15}
\definecolor{dark-blue}{rgb}{0.15,0.15,0.6}
\definecolor{dark-green}{rgb}{0.15,0.6,0.15}

\hypersetup{
    colorlinks, linkcolor=Blue,
    citecolor=Blue, urlcolor=Blue
}

\usepackage[nameinlink,capitalise,noabbrev]{cleveref}
\usepackage[hyphenbreaks]{breakurl}

\numberwithin{equation}{section}
\usepackage[all]{xy}
\xyoption{line}
\usepackage{graphicx}
\usepackage{mathtools}

\newtheorem{ThmAlpha}{Theorem}

\swapnumbers 

\newtheorem{Thm}[equation]{Theorem}
\newtheorem*{Thm*}{Theorem}
\newtheorem{Prop}[equation]{Proposition}
\newtheorem{Lem}[equation]{Lemma}
\newtheorem{Cor}[equation]{Corollary}

\newtheorem*{Conj*}{Conjecture}

\newtheorem*{Que*}{Question}

\theoremstyle{remark}
\newtheorem{Def}[equation]{Definition}

\newtheorem{Not}[equation]{Notation}
\newtheorem{Exa}[equation]{Example}

\newtheorem{Hyp}[equation]{Hypothesis}

\newtheorem{Rem}[equation]{Remark}

\tikzset{
    labelrotatebelow/.style={anchor=north, rotate=90, inner sep=1.0mm}
}
\tikzset{
    labelrotateabove/.style={anchor=south, rotate=90, inner sep=1.0mm}
}

\newcommand{\nc}{\newcommand}
\nc{\dmo}{\DeclareMathOperator}

\renewcommand{\emptyset}{\varnothing}

\nc{\Beren}[1]{{\color{MidnightBlue}#1}}
\nc{\Drew}[1]{{\color{OliveGreen}#1}}
\nc{\Tobi}[1]{{\color{Orange}#1}}
\nc{\Dout}[1]{\Drew{\sout{#1}}}
\nc{\Bout}[1]{\Beren{\sout{#1}}}
\nc{\Tout}[1]{\Tobi{\sout{#1}}}

\usepackage{todonotes}

\nc{\overbar}[1]{\mkern 1.5mu\overline{\mkern-1.5mu#1\mkern-1.5mu}\mkern 1.5mu}
\nc{\haux}{\mathrm{h}}
\nc{\supph}{\supp^\haux}
\nc{\Supph}{\Supp^\haux}
\nc{\Spch}{\Spc^\haux}

\nc{\kappaaux}{g}
\nc{\kappaCh}{{\kappaaux(\cat C_h)}}
\nc{\kappam}{{\kappaaux({\frak m})}}
\nc{\kappaP}{{\kappaaux(\cat P)}}
\nc{\kappaQ}{{\kappaaux(\cat Q)}}
\nc{\kappaCP}{{\kappaaux_{\cat C}(\cat P)}}
\nc{\kappaDP}{{\kappaaux_{\cat D}(\cat P)}}
\nc{\kappaCQ}{{\kappaaux_{\cat C}(\cat Q)}}
\nc{\kappaDQ}{{\kappaaux_{\cat D}(\cat Q)}}
\nc{\kappaphiB}{{\kappaaux(\phi(\cat B))}}
\nc{\kappaphiQ}{{\kappaaux(\varphi(\cat Q))}}

\dmo{\Sub}{Sub}
\dmo{\Tor}{Tor}
\nc{\SpEn}{\cat S_{E(n)}}
\nc{\SpEnf}{\cat S_n}
\nc{\Lcomp}{L^{\mathrm{com}}} 
\nc{\Ucomp}{U^{\mathrm{com}}}
\nc{\Loco}[1]{\Loc_{\otimes}\hspace{-0.3ex}\langle #1 \rangle}
\nc{\bbullet}{{\scriptscriptstyle\hspace{-1pt}\bullet}}
\nc{\bullett}{{\scriptscriptstyle\bullet}\hspace{-1pt}}
\nc{\LF}{L\hspace{-0.2ex}F}
\nc{\SpG}{\Sp^G}
\nc{\EG}{\bbE_G}
\nc{\DEG}{\Der(\EG)}
\nc{\DE}{\Der(\bbE)}
\nc{\Prst}{{\cat P}\mathrm{r^{st}}}
\nc{\Mack}[2]{\mathrm{Mack}_{#1}(#2)}
\nc{\SC}{S\cat C}
\dmo{\fin}{{fin}}
\dmo{\DM}{DM}
\dmo{\fp}{fp}
\nc{\DMQ}{\DM_Q}
\dmo{\DerKal}{DMack}
\dmo{\Der}{D}
\dmo{\DMot}{DMot}
\dmo{\rmH}{H}
\dmo{\piu}{\underline{\pi}}
\dmo{\Sphere}{\mathbb{S}}
\nc{\HA}{{\rmH \hspace{-0.2em}\bbA}}
\nc{\HZ}{{\rmH \hspace{-0.2em}\bbZ}}
\nc{\HZbar}{{\rmH \hspace{-0.2em}\underline{\bbZ}}}
\nc{\Fp}{{\bbF_{\hspace{-0.1em}p}}}
\nc{\HFp}{{\rmH \hspace{-0.15em}\bbF_{\hspace{-0.1em}p}}}
\nc{\HQ}{{\rmH \hspace{-0.15em}\bbQ}}
\nc{\DHZG}{\Der(\HZ_G)}
\nc{\DHZH}{\Der(\HZ_H)}
\nc{\DHZK}{\Der(\HZ_K)}
\nc{\DHZGN}{\Der(\HZ_{G/N})}
\nc{\DHZGG}{\Der(\HZ_{G/G})}
\nc{\DHZCp}{\Der(\HZ_{C_p})}
\nc{\DHZGprime}{\Der(\HZ_{G'})}
\nc{\DHZ}{\Der(\HZ)}
\nc{\frakp}{\mathfrak{p}}
\nc{\frakq}{\mathfrak{q}}
\nc{\frakS}{\mathfrak{S}}
\nc{\frakT}{\mathfrak{T}}
\nc{\Z}{\mathbb{Z}}
\nc{\SSG}{\text{sSet}_*^G}
\nc{\sSet}{\text{sSet}}

\dmo{\csupp}{csupp}
\dmo{\Con}{Conj}
\dmo{\Id}{Id}
\dmo{\Loc}{Loc}
\dmo{\rmK}{\textrm{\rm K}}
\dmo{\Spc}{Spc}
\dmo{\thick}{thick}
\nc{\thickt}[1]{\thick_\otimes\langle #1 \rangle}
\dmo{\cone}{cone}
\dmo{\End}{End}
\dmo{\Mor}{Mor}
\dmo{\Hom}{Hom}
\dmo{\id}{id}
\dmo{\incl}{incl}
\dmo{\Img}{Im}
\dmo{\im}{im}
\dmo{\Ker}{Ker}
\dmo{\ind}{ind}
\dmo{\CoInd}{coind}
\dmo{\res}{res}
\dmo{\infl}{infl}
\dmo{\triv}{triv}
\dmo{\Tel}{Tel} 
\dmo{\grMod}{grMod}%
\dmo{\Mod}{Mod}%
\dmo{\opname}{op}
\dmo{\SH}{SH}
\dmo{\smallb}{b}
\dmo{\Spec}{Spec}
\dmo{\supp}{supp}
\dmo{\Supp}{Supp}
\nc{\SHc}{{\SH^c}}
\nc{\SHp}{{\SH_{(p)}}}
\nc{\SHcp}{{\SH^c_{(p)}}}
\nc{\SHG}{\SH(G)}
\nc{\SHGp}{\SH(G)_{(p)}}
\nc{\SHGc}{\SHG^c}
\nc{\SHGcp}{\SHG^c_{(p)}}
\nc{\quadtext}[1]{\quad\textrm{#1}\quad}
\nc{\qquadtext}[1]{\qquad\textrm{#1}\qquad}
\nc{\adj}{\dashv}
\nc{\adjto}{\rightleftarrows}
\nc{\bbL}{\mathbb{L}}
\nc{\bbA}{\mathbb{A}}
\nc{\bbE}{\mathbb{E}}
\nc{\bbN}{\mathbb{N}}
\nc{\bbQ}{\mathbb{Q}}
\nc{\bbZ}{\mathbb{Z}}
\nc{\bbF}{\mathbb{F}}
\nc{\cat}[1]{\mathscr{#1}}
\nc{\ie}{{\sl i.e.}, }
\nc{\into}{\mathop{\rightarrowtail}}
\nc{\inv}{^{-1}}
\nc{\isoto}{\mathop{\overset{\sim}\to}}
\nc{\isotoo}{\mathop{\overset{\sim}\too}}
\nc{\onto}{\mathop{\twoheadrightarrow}}
\nc{\too}{\mathop{\longrightarrow}\limits}
\nc{\mapstoo}{\longmapsto}
\nc{\adh}[1]{\overline{#1}}
\nc{\adhpt}[1]{\adh{\{#1\}}}
\nc{\aka}{{a.\,k.\,a.}\ }
\nc{\calF}{\mathcal{F}}
\nc{\eg}{{\sl e.\,g.}}
\nc{\Homcat}[1]{\Hom_{\cat #1}}
\nc{\hook}{\hookrightarrow}
\nc{\ideal}[1]{\langle #1\rangle}
\nc{\ihom}{{\underline{\hom}}}
\nc{\Mid}{\,\big|\,}
\nc{\MMod}{\,\text{-}\Mod}%
\nc{\GrMMod}{\,\text{-}\grMod}%
\nc{\op}{^{\opname}}
\nc{\oto}[1]{\overset{#1}\to}
\nc{\otoo}[1]{\overset{#1}{\,\too\,}}
\nc{\sminus}{\!\smallsetminus\!}
\nc{\poplus}[1]{^{\oplus #1}}%
\nc{\potimes}[1]{^{\otimes #1}}
\nc{\sbull}{{\scriptscriptstyle\bullet}}
\nc{\SET}[2]{\big\{\,#1\Mid#2\,\big\}}
\nc{\SpcK}{\Spc(\cat K)}
\nc{\then}{\Rightarrow}
\nc{\unit}{\mathbb{1}}
\nc{\xra}{\xrightarrow}
\nc{\phigeom}[1]{\widetilde{\Phi}^{#1}}
\dmo{\Oname}{O}
\dmo{\proper}{proper}
\dmo{\lenormal}{\unlhd}
\dmo{\lnormal}{\lhd}
\nc{\normal}{\trianglelefteq}
\nc{\Op}{\Oname^p}
\nc{\Oq}{\Oname^q}
\dmo{\Sp}{Sp}
\dmo{\Ho}{Ho}
\dmo{\Fin}{Fin}
\dmo{\add}{add}
\dmo{\Fun}{Fun}
\dmo{\Ext}{Ext}
\dmo{\CAlg}{CAlg}
\dmo{\CMon}{CMon}
\dmo{\CC}{\cat C} 
\dmo{\DD}{\cat D}
\dmo{\OO}{\mathcal{O}}
\dmo{\Map}{Map}
\dmo{\Span}{Span}
\dmo{\N}{N}
\dmo{\Cat}{Cat}
\dmo{\colim}{colim}
\dmo{\hocolim}{hocolim}
\dmo{\Ch}{Ch}
\dmo{\A}{\mathbb{A}^{eff}}
\nc{\AGeff}{\mathbb{A}_G^{\mathrm{eff}}}
\nc{\BGeff}{\mathcal{B}_G^{\mathrm{eff}}}
\nc{\BG}{{\mathcal{B}_G}}
\nc{\NBGeff}{{\N}{\BGeff}}
\dmo{\Ab}{Ab}
\dmo{\Set}{Set}
\dmo{\ev}{ev}
\dmo{\Spcl}{Spcl}
\nc{\Funadd}{\Fun_{\add}}
\dmo{\proj}{proj}
\dmo{\cof}{cof}

\dmo{\Coideal}{Coideal}
\dmo{\gen}{gen}

\nc{\mT}{\kern-0.5em\mod\kern-0.1em\text{-}\cat{T}^c}
\nc{\mTc}{\kern-0.5em\mod\kern-0.1em\text{-}\cat{T}^c}
\nc{\MTc}{\Mod\kern-0.1em\text{-}\cat{T}^c}
\nc{\MT}{\Mod\kern-0.1em\text{-}\cat{T}}
\newcounter{enum-resume-hack}

\begin{document}

\title[The comparison between homological and triangular support]{Stratification and the comparison between homological and tensor triangular support}

\author{Tobias Barthel}
\author{Drew Heard}
\author{Beren Sanders}
\date{\today}

\makeatletter
\patchcmd{\@setaddresses}{\indent}{\noindent}{}{}
\patchcmd{\@setaddresses}{\indent}{\noindent}{}{}
\patchcmd{\@setaddresses}{\indent}{\noindent}{}{}
\patchcmd{\@setaddresses}{\indent}{\noindent}{}{}
\makeatother

\address{Tobias Barthel, Max Planck Institute for Mathematics, Vivatsgasse 7, 53111 Bonn, Germany}
\email{tbarthel@mpim-bonn.mpg.de}
\urladdr{\href{https://sites.google.com/view/tobiasbarthel/home}{https://sites.google.com/view/tobiasbarthel/home}}

\address{Drew Heard, Department of Mathematical Sciences, Norwegian University of Science and Technology, Trondheim}
\email{drew.k.heard@ntnu.no}
\urladdr{\href{https://folk.ntnu.no/drewkh/}{https://folk.ntnu.no/drewkh/}}

\address{Beren Sanders, Mathematics Department, UC Santa Cruz, 95064 CA, USA}
\email{beren@ucsc.edu}
\urladdr{\href{http://people.ucsc.edu/~beren/}{http://people.ucsc.edu/$\sim$beren/}}

\begin{abstract}
We compare the homological support and tensor triangular support for `big' objects in a rigidly-compactly generated tensor triangulated category.  We prove that the comparison map from the homological spectrum to the tensor triangular spectrum is a bijection and that the two notions of support coincide whenever the category is stratified, extending work of Balmer. Moreover, we clarify the relations between salient properties of support functions and exhibit counter-examples highlighting the differences between homological and tensor triangular support.
\end{abstract}


\thanks{The first-named author would like to thank the Max Planck Institute for Mathematics for its hospitality. The second-named author is supported by grant number TMS2020TMT02 from the Trond Mohn Foundation. The third-named author is supported by NSF grant~DMS-1903429.}

\maketitle

\vspace{-3ex}
{
\hypersetup{linkcolor=black}
\tableofcontents
}
\vspace{-3ex}

\section{Introduction}

Support functions for the objects of a given tensor triangulated category have been employed in a variety of contexts to establish classification theorems.  Prominent historical examples include the support varieties of modular representation theory \cite{Carlson90,BensonCarlsonRickard95,BensonCarlsonRickard97}, chromatic support in stable homotopy theory \cite{Hopkins87,HopkinsSmith98}, and notions of support for derived categories in algebraic geometry \cite{Neeman92a, Thomason97}. Balmer \cite{Balmer05a} unified these developments by constructing a universal notion of support $(\Spc(\cat K), \supp_{\cat K})$ for any essentially small tensor triangulated category~$\cat K$. Taking values in the Balmer spectrum $\Spc(\cat K)$, this universal notion of support classifies the thick tensor-ideals of $\cat K$, and provides an abstract conceptual generalization of the specific support theories that arise in different subjects.

More recently, motivated by the elusive search for residue fields in tensor triangular geometry, \cite{BalmerKrauseStevenson19} have introduced the so-called homological residue fields of~$\cat K$.  They are parametrized by a new space $\Spch(\cat K)$ called the homological spectrum, and they have been used by Balmer \cite{Balmer20_nilpotence} to introduce a new notion of support $(\Spch(\cat K),\supph_{\cat K})$ for the objects of $\cat K$, which complements the universal notion of support $(\Spc(\cat K),\supp_{\cat K})$.  By virtue of the universal property, there is a canonical continuous map $\phi:\Spch(\cat K) \to \Spc(\cat K)$ such that $\supph_{\cat K} = \phi^{-1}(\supp_{\cat K})$.  This comparison map $\phi$ is surjective (under very mild hypotheses) and turns out to be bijective in all known examples \cite[Section 5]{Balmer20_nilpotence}. This motivates the following conjecture:\footnote{Personal communication; see also \cite[Rem.~5.15]{Balmer20_nilpotence}}

\begin{Conj*}[Balmer]
	The map $\phi\colon \Spch(\cat K) \to \Spc(\cat K)$ is always a bijection.
\end{Conj*}

In fact, this conjecture admits a purely point-set topological reformulation without reference to the triangular spectrum, as we show in \cref{prop:homeo}:

\begin{ThmAlpha}
The comparison map $\phi$ is a bijection if and only if $\Spch(\cat K)$ is a $T_0$-space; and if that is the case, then $\phi$ is a homeomorphism.
\end{ThmAlpha}

However, the comparison between the homological and triangular perspectives on tensor triangular geometry runs deeper, since even in the case that $\phi$ is a homeomorphism, they in general afford different support theories, as we discuss next.

Indeed, in many contexts the tt-category $\cat K$ arises as the full subcategory of compact objects in a rigidly-compactly generated tt-category $\cat T$.  This leads to the problem of finding a suitable construction of support for `big' objects in $\cat T$ which extends the universal notion of support for compact objects.  A primary motivation is to use such a notion of big support to stratify the category $\cat T$, that is, to classify the localizing tensor-ideals of $\cat T$, much as Balmer's universal notion of support classifies the thick tensor-ideals of $\cat K=\cat T^c$.  Different approaches have been proposed. In a seminal series of papers \cite{BensonIyengarKrause08,BensonIyengarKrause11b,BensonIyengarKrause11a} building on \cite{HoveyPalmieriStrickland97}, Benson, Iyengar and Krause have developed a theory of `big' support in terms of a suitable ring action on $\cat T$, which led to important applications in modular representation theory.  On the other hand, Balmer and Favi \cite{BalmerFavi11} give a construction of big support $(\Spc(\cat T^c),\Supp_{\cat T})$ in the setting of tensor triangular geometry.  Both approaches admit a uniform generalization through the work of Stevenson \cite{Stevenson13}; see also \cite{bhs1}.

These approaches to big support fundamentally depend on some noetherian hypothesis.  For example, the Balmer--Favi notion of support $\Supp_{\cat T}$ does not provide an extension of the universal support on $\cat T^c$ to the whole of $\cat T$ without some such hypothesis.  Recently, Balmer \cite{Balmer20_bigsupport} has extended homological support to big objects.  The resulting notion of support $(\Spch(\cat T^c),\Supph_{\cat T})$ does not require any noetherian hypotheses and extends the pull-back $\supph_{\cat T^c} = \phi^{-1}(\supp_{\cat T^c})$ of the universal support to the whole of $\cat T$.

In this paper, we study the relationship between the homological spectrum $\Spch(\cat T^c)$ and the Balmer spectrum $\Spc(\cat T^c)$ via the comparison map $\phi$ as well as the relation between the notions of big support, $\Supph_{\cat T}$ and $\Supp_{\cat T}$, which inhabit these spaces. Our main result is \cref{thm:homological}:

\begin{ThmAlpha}\label{thm:A}
	Let $\cat T$ be a rigidly-compactly generated tt-category whose spectrum $\Spc(\cat T^c)$ is weakly noetherian.  If $\cat T$ is stratified, then the following statements hold:
    \begin{enumerate}
        \item the comparison map $\phi\colon \Spch(\cat T^c) \to \Spc(\cat T^c)$ is a homeomorphism;
        \item $\Supph_{\cat T}(t) = \phi^{-1}(\Supp_{\cat T}(t))$ for all $t \in \cat T$; and
        \item both $\Supp$ and $\Supph$ detect trivial objects and satisfy the tensor product property.
    \end{enumerate}
\end{ThmAlpha}

This adds several new examples to the list of tt-categories for which the comparison map $\phi$ is known to be a homeomorphism, and to which the techniques of \cite{Balmer20_nilpotence} do not readily apply. (See the examples in \cref{sec:examples}.) En route to proving the theorem, we establish a series of partial comparison results. These establish a hierarchy among different properties of support, as follows:

\begin{ThmAlpha}\label{thm:B}
	Let $\cat T$ be a rigidly-compactly generated tt-category whose spectrum $\Spc(\cat T^c)$ is weakly noetherian. 
    \begin{enumerate}
        \item For any $t \in \cat T$, we have $\phi(\Supph_{\cat T}(t)) \subseteq \Supp_{\cat T}(t)$. In particular, if $\Supph_{\cat T}$ detects trivial objects, then so does $\Supp_{\cat T}$.
        \item Consider the following three statements:
            \begin{enumerate}
                \item[(1)] $\Supph_{\cat T}(t) = \phi^{-1}(\Supp_{\cat T}(t))$ for all $t \in \cat T$.
                \item[(2)] $\Supp_{\cat T}(s \otimes t) = \Supp_{\cat T}(s) \cap \Supp_{\cat T}(t)$ for all $s,t\in \cat T$.
                \item[(3)] The comparison map $\phi\colon \Spch(\cat T^c) \to \Spc(\cat T^c)$ is a bijection.
            \end{enumerate}
            Then $(1) \Rightarrow (2) \Rightarrow (3)$. If $\Supph_{\cat T}$ detects trivial objects, then $(3) \Rightarrow (1)$.
    \end{enumerate}
\end{ThmAlpha}

The proof of \cref{thm:B} will be assembled at the end of \cref{sec:comparasion}. In \cref{sec:examples}, we turn to specific applications. In particular, we show that the converse of (a) as well as the implication $(3) \Rightarrow (1)$ in \cref{thm:B} fail in general if $\Supph_{\cat T}$ does not detect trivial objects. In fact, based on work of Neeman \cite{Neeman00}, we exhibit an example of a tt-category $\cat T$ for which the homological and the tensor triangular spectrum coincide, but which contains objects $t \in \cat T$ with $\Supph_{\cat T}(t) \subsetneq \Supp_{\cat T}(t)$.  Furthermore, we use our results to clarify the relation between different notions of support for derived categories of non-noetherian commutative rings, in derived algebra, and in chromatic homotopy theory. 

Balmer's conjecture remains open in full generality.  Either answer would be interesting.  A counterexample would provide the possibility that the homological spectrum could serve as a home for theories of support which have a better chance of classifying localizing tensor-ideals --- at least in some non-noetherian settings.  A positive answer would unify the two kinds of ``spectra'' for tensor triangulated categories introduced by Balmer and establish in full generality the equality of two quite different approaches to defining support.  Our results demonstrate that a counterexample to Balmer's conjecture must necessarily come from non-stratified categories.

\subsection*{Acknowledgements}

We thank Scott Balchin, Paul Balmer and David Rubinstein for useful conversations.  In particular, we are grateful for Paul Balmer's suggestion to use the half-tensor formula in the proof of \cref{lem:support_injectives}.
We also thank Changhan Zou and an anonymous referee for their careful reading of a previous version of the manuscript.

\medskip
\section{Support for rigidly-compactly generated tt-categories}

In this section, we introduce an abstract notion of support function for arbitrary objects in a rigidly-compactly generated tt-category and give several known examples. 

\begin{Def}\label{defn:supportfunction}
	A \emph{support function} for a rigidly-compactly generated tt-category~$\cat T$ is a pair $(X,\frakS)$ consisting of a set $X$ and a function $\frakS\colon \mathrm{Ob}(\cat T) \to \mathcal{P}(X)$ assigning a subset of $X$ to every object of $\cat T$, subject to the following conditions:
    \begin{enumerate}
        \item \label{it:support-zero} $\frakS(0) = \emptyset$ and $\frakS(\unit)=X$;
        \item $\frakS(\Sigma t) = \frakS(t)$ for  all $t \in \cat T$;
 	\item \label{it:support-triangle} $\frakS(t_3) \subseteq \frakS(t_1) \cup \frakS(t_2)$ for any exact triangle $t_1 \to t_2 \to t_3 \to \Sigma t_1$ in $\cat T$;
        \item $\frakS(\coprod_{i \in I}t_i) = \bigcup_{i \in I}\frakS(t_i)$ for any family $\{t_i\}_{i \in I}$ of objects in $\cat T$;
        \item $\frakS(t_1 \otimes t_2) \subseteq \frakS(t_1) \cap  \frakS(t_2)$ for any $t_1,t_2 \in \cat T$.
    \end{enumerate}
\end{Def}

\begin{Rem}\label{rem:supportdata}
	This notion of support function is inspired by Balmer's definition of support data for an essentially small tt-category \cite{Balmer05a}.  However, the restriction of a support function $(X,\frakS)$ in the sense of \cref{defn:supportfunction} to the full subcategory $\cat T^c$ of compact objects in $\cat T$ is not a support datum in the sense of \cite[Definition~3.1]{Balmer05a}.  In order for this to be the case, we would need to additionally demand that the set~$X$ is equipped with a topology such that the restriction $\frakS{|}_{\cat T^c}$ satisfies:
    \begin{enumerate}
        \item[(f)] $\frakS{|}_{\cat T^c}\colon \cat T^c \to \mathcal{P}(X)$ takes values in closed subsets of $X$;
        \item[(g)] $\frakS(s \otimes t) = \frakS(s) \cap \frakS(t)$ for all $s,t \in \cat T^c$.
    \end{enumerate}
	If that is the case, then by the universal property of the Balmer spectrum \cite[Thm.~3.2]{Balmer05a}, there is a unique morphism of support data $(X,\frakS) \to (\Spc(\cat T^c),\supp_{\cat T^c})$.  In other words, there is a unique continuous map $f:X \to \Spc(\cat T^c)$ such that 
	\[
		f^{-1}(\supp_{\cat T^c}(x)) = \frakS(x)
	\]
	for all compact objects $x \in \cat T^c$. 
\end{Rem}

\begin{Rem}
	We now recall the Balmer--Favi support function which (under some noetherian hypotheses) extends the universal notion of support $\supp_{\cat T^c}$ from $\cat T^c$ to all of $\cat T$.  A more extensive discussion can be found in~\cite{bhs1}.
\end{Rem}
\begin{Def}\label{rem:weakly-noetherian}
	A point $x$ in a spectral space $X$ is said to be \emph{visible} if its closure $\overbar{\{x\}}$ is a Thomason subset of $X$, and is said to be \emph{weakly visible} if there exist two Thomason subsets $Y_1,Y_2 \subseteq X$ such that $\{x\} = Y_1 \cap Y_2^c$. A space $X$ is said to be \emph{weakly noetherian} if every point of $X$ is weakly visible. \end{Def}
    \begin{Rem} This terminology is justified because a visible point is weakly visible and a spectral space is noetherian if and only if each of its points is visible (\cite[Cor.~7.14]{BalmerFavi11}). An example of a spectral space which is not weakly noetherian is the Balmer spectrum of the category of finite $p$-local spectra. On the other hand, the Balmer spectrum of the category of finite rational $G$-spectra is always weakly noetherian but it is not noetherian in general when the compact Lie group $G$ is not finite. Another example of a weakly noetherian space which is not noetherian is the spectrum $\Spec(R)$ of a non-noetherian absolutely flat ring~$R$ (such as an infinite product of fields).
	See \cite[Remark~4.3]{Stevenson14} and \cite[Example~2.5]{bhs1}.
\end{Rem}

\begin{Exa}[Balmer--Favi support]\label{ex:balmer-favi-support}
	Let $\cat T$ be a rigidly-compactly generated tt-category whose spectrum $\Spc(\cat T^c)$ is weakly noetherian.  Under the latter hypothesis, Balmer's universal notion of support $\supp_{\cat T^c}$ for $\cat T^c$ admits an extension to a support function on all of $\cat T$.  This notion of support for big objects was introduced by Balmer--Favi in \cite[Section 7]{BalmerFavi11}.  They construct for every (weakly visible) point $\cat P \in \Spc(\cat T^c)$ a $\otimes$-idempotent $\kappaP \in \cat T$ and then define
	\[
		\Supp_{\cat T}(t) \coloneqq \{\cat P \in \Spc(\cat T^c) \mid \kappaP \otimes t \neq 0 \}
	\]
	for any $t \in \cat T$.  See \cite[Section~2]{bhs1} for details. This defines a support function $\Supp_{\cat T}$ for $\cat T$ taking values in $\Spc(\cat T^c)$ with the property that $\Supp_{\cat T}{|}_{\cat T^c} = \supp_{\cat T^c}$.  This was shown in \cite[Proposition~7.17]{BalmerFavi11} under the hypothesis that $\Spc(\cat T^c)$ is noetherian, while the general case was established in Remark 2.12 and Lemma 2.13 of \cite{bhs1}.
\end{Exa}

\begin{Rem}
	For our purposes, the next most significant example is the homological support function introduced by Balmer \cite{Balmer20_bigsupport}.  We briefly recall the construction; more details can be found in \cite{Balmer20_bigsupport} and \cite{BalmerKrauseStevenson19}.
\end{Rem}

\begin{Exa}[Homological support]\label{ex:supph}
	Let $\cat T$ be a rigidly-compactly generated tt-category. No assumptions on $\Spc(\cat T^c)$ are required. Let $h \colon \cat T \to \MTc$ denote the restricted Yoneda functor from~$\cat T$ to the Grothendieck abelian category $\cat A \coloneqq \MTc$ of right $\cat T^c$-modules, i.e., additive functors $M \colon (\cat T^c)\op \to \Ab$.
	Let $\cat A^{\fp}$ denote the full subcategory of $\cat A$ consisting of the finitely presented modules.  Every Serre $\otimes$-ideal $\cat B \subseteq \cat A^{\fp}$ generates a localizing Serre $\otimes$-ideal $\langle \cat B\rangle$ of $\cat A$, and we can consider the Gabriel quotient $ \cat A/\langle \cat B \rangle$.  We let $\overbar{h}_{\cat B} \colon \cat T \to \cat A/\langle \cat B \rangle$ denote the composite $\cat T \to \cat A \twoheadrightarrow \cat A/\langle \cat B \rangle $.  As explained in \cite[Sections 2--3]{BalmerKrauseStevenson19}, there is a corresponding pure-injective object $E_{\cat B} \in \cat T$ such that 
	\[
		\langle \cat B \rangle = \Ker(h(E_{\cat B}) \otimes -).
	\]
	Moreover, for $t \in \cat T$, we have $t \otimes E_{\cat B} = 0$ if and only if $\overbar{h}_{\cat B}(t) = 0$.  The \emph{homological spectrum} $\Spch(\cat T^c)$ is the set of maximal Serre $\otimes$-ideals $\cat B \subset \cat{A}^{\fp}$.  Its points are the \emph{homological primes} of $\cat T^c$.  The homological support $\Supph_{\cat T}(\cat T) \subseteq \Spch(\cat T^c)$ of an object $t \in \cat T$ is defined by
	\[
		\Supph_{\cat T}(t) \coloneqq \SET{ \cat B \in \Spch(\cat T^c) }{ [t,E_{\cat B}] \neq 0}
	\]
	where $[-,-]$ denotes the internal hom. By \cite[Thm.~2.1]{Balmer20_bigsupport}, the homological support $\Supph_{\cat T}$ defines a support function for $\cat T$.  We also equip $\Spch(\cat T^c)$ with a topology by taking the homological supports of compact objects as a basis of closed sets (see \cite[Rem.~3.4]{Balmer20_nilpotence}).
\end{Exa}

\begin{Rem}\label{rem:naivesupph}
	There is also a notion of `naive' homological support (see \cite[Rem.~4.6]{Balmer20_nilpotence}) defined by testing with $-\otimes E_{\cat B}$ rather than with $[-,E_{\cat B}]$.  By \cite[Prop.~3.10]{Balmer20_bigsupport}, $[t,E_{\cat B}] \neq 0$ implies $t \otimes E_{\cat B} \neq 0$ for any $t \in \cat T$, so the `naive' homological support contains the homological support.  It is an open question whether these two notions of homological support coincide in general (but see \cref{rem:naive_support} below).
\end{Rem}

\begin{Exa}[BIK support]\label{ex:bik}
	Another prominent class of support functions arise from the action of a
	(graded) commutative noetherian ring $R$ on a compactly generated
	tt-category.  In the presence of such an action, Benson, Iyengar and Krause
	\cite{BensonIyengarKrause08} have constructed a support function
	$\Supp_{R}$ for $\cat T$ which takes values in the (homogeneous) Zariski
	spectrum $\Spec(R)$,
	or rather the subset $\Supp_{R}(\unit) \subseteq \Spec(R)$. 
	We refer to \cite{BensonIyengarKrause08} for the
	details.  If $\cat T$ is stratified by the action of~$R$ in the sense of
	\cite{BensonIyengarKrause11b} then 
	$\Supp_{R}(\unit) \cong \Spc(\cat T^c)$
	and
	under this identification the BIK notion of support coincides with the
	Balmer--Favi notion of support: $\Supp_R = \Supp_{\cat T}$.  This is
	established in \cite[Cor.~7.11]{bhs1}.\end{Exa}

\begin{Exa}[Bousfield lattice support]
	In \cite{IyengarKrause13}, Iyengar and Krause use the Bousfield lattice of
	a rigidly-compactly generated tt-category $\cat T$ to construct another
	support function $\Supp_{\text{BL}}\colon \cat T \to \mathrm{Sp}(\cat T)$.
	The target $\Sp(\cat T)$ is the space corresponding via Stone duality to
	the distributive lattice of idempotent Bousfield classes of $\cat T$. The
	verification that  $\Supp_{\text{BL}}$ defines a support function is \cite[Prop.~6.3]{IyengarKrause13}.
	Since the
	Bousfield class of any compact object is idempotent, $\Supp_{\text{BL}}$
	satisfies Condition~(f) of \cref{rem:supportdata} as well.  It follows
	that, up to passage to the opposite topology on $\mathrm{Sp}(\cat T)$, the
	restriction of $\Supp_{\text{BL}}$ to $\cat T^c$ is a support datum in the
	sense of Balmer; cf.~\cite[Prop.~7.9]{IyengarKrause13}. \end{Exa}

\section{Comparison with homological support}\label{sec:comparasion}

In this section we study the relationship between the Balmer--Favi notion of support (\cref{ex:balmer-favi-support}) and Balmer's homological support (\cref{ex:supph}) for big objects in a rigidly-compactly generated tt-category $\cat T$. In the next section, we will show that the two notions agree whenever $\cat T$ is stratified (see \cref{thm:homological}). Our approach to this result is not geodesic, however, as we include some partial comparison results along the way.

\begin{Hyp}
	Throughout this section we assume that $\cat T$ is a rigidly-compactly generated tt-category whose spectrum $\Spc(\cat T^c)$ is weakly noetherian (\cref{rem:weakly-noetherian}).
\end{Hyp}
\begin{Def}\label{ter:suppproperties}
	We say that a support datum $(X,\frakS)$ for $\cat T$
    \begin{enumerate}
        \item has the detection property if, for any $t \in \cat T$, $\frakS(t) = \emptyset$ implies $t = 0$;
        \item satisfies the tensor product formula if $\frakS(s \otimes t) = \frakS(s) \cap \frakS(t)$ for any $s,t \in \cat T$.
    \end{enumerate}
\end{Def}

\begin{Rem}\label{rem:tensor-nilpotence}
	In the presence of non-trivial $\otimes$-nilpotent objects in $\cat T$, a support function cannot satisfy both the detection property and the tensor product formula, because
	\[
		\emptyset = \frakS(t^{\otimes n}) = \frakS(t)
	\]
	forces $t = 0$ for any $t\in \cat T$ and all $n \ge 1$.
\end{Rem}

\begin{Rem}\label{rem:supph-tensor-product}
	By \cite[Thm.~1.2]{Balmer20_bigsupport}, the homological support (\cref{ex:supph}) always satisfies the tensor product formula. 
\end{Rem}

\begin{Rem}\label{rem:idempotents}
	Recall from \cite{BalmerFavi11} that smashing localizations of $\cat T$ correspond to 
idempotent triangles in $\cat T$; that is, exact triangles
        \[
            e \to \unit \to f \to \Sigma e
        \]
    with the property that $e \otimes f=0$.  It follows that the objects $e$ and $f$ are 
	tensor-idempotents
	($e\otimes e \cong e$ and $f \otimes f \cong f$) and that the functor $f \otimes - \colon \cat T \to \cat T$ is a smashing localization.
%
	For example, 
    given a Thomason subset $Y \subseteq \Spc(\cat T^c)$, with corresponding thick $\otimes$-ideal $\cat T^c_Y = \SET{ x\in \cat T^c}{\supp_{\cat T^c}(x) \subseteq Y}$, there is an associated idempotent triangle 
        \[ 
            e_Y \to \unit \to f_Y \to \Sigma e_Y
        \]
    in $\cat T$ such that $\Ker(f_Y \otimes -)=e_Y \otimes \cat T = \Loco{e_Y} = \Loc\langle \cat T^c_Y\rangle$. 
\end{Rem}

\begin{Rem}\label{rem:homological_spectrum_map}
	There is a continuous map
	\[
		\begin{split}
		\phi \colon \Spch(\cat T^c) &\to \Spc(\cat T^c) \\
		\cat B &\mapsto  h^{-1}(\cat B)
		\end{split}
	\] 
	constructed in \cite[Rem.~3.4]{Balmer20_nilpotence}. Since $\cat T^c$ is rigid, the map $\phi$ is surjective by \cite[Cor.~3.9]{Balmer20_nilpotence}.  Given the evidence collected in \cite[Sec.~5]{Balmer20_nilpotence}, it took Balmer ``nerves of steel not to conjecture'' that~$\phi$ is in fact a bijection \cite[Rem.~5.15]{Balmer20_nilpotence}.
\end{Rem}
\begin{Lem}\label{lem:support_injectives}
	If $\cat B \in \Spch(\cat T^c)$ then $\Supp_{\cat T}(E_{\cat B}) = \{ \phi(\cat B)\}$.
\end{Lem}

\begin{proof}
	Let $\cat P := \phi(\cat B) \in \Spc(\cat T^c)$.  We claim that $\kappaP \otimes E_{\cat B} \neq 0$ and ${\kappaQ \otimes E_{\cat B} = 0}$ for $\cat Q \ne \cat P$.  Equivalently (by the discussion in \cref{ex:supph}), $\overbar{h}_{\cat B}(\kappaP) \ne 0$ and $\overbar{h}_{\cat B}(\kappaQ) = 0$ for $\cat Q \neq \cat P$.

	We first show that $\overbar{h}_{\cat B}(\kappaP) \ne 0$. By \cite[Cor.~3.6]{Balmer20_nilpotence} (applied with $\cat J = \cat P$) and \cite[Prop.~3.11(a)]{bhs1} we can assume that $\cat T$ is local and that $\cat P = \frak m$ is the unique closed point of $\Spc(\cat T^c)$. In this case, $\kappam \simeq e_{\frak m}$. Moreover, $\overbar{h}_{\cat B}(e_{\frak m}) \ne 0$ follows from $\overbar{h}_{\cat B}(f_{\frak m}) = 0$, which is shown in \cite[Cor.~4.14]{BalmerKrauseStevenson19}. 

	Now consider $\cat Q \in \gen(\cat P) \setminus \{\cat P\}$, where $\gen(\cat P)$ denotes the set of generalizations of $\cat P$.
	We claim that $\overbar{h}_{\cat B}(\kappaQ) = 0$.  Again using \cite[Cor.~3.6]{Balmer20_nilpotence} and \cite[Prop.~3.11(a)]{bhs1} we can assume that $\cat T$ is local and that $\cat Q \ne \frak m$. In this case, $\kappaQ = e_{Y_1} \otimes f_{Y_2}$ for Thomason subsets $Y_1,Y_2$ such that $\{ \cat Q \} = Y_1 \cap Y_2^c$, and moreover $Y_2 \ne \emptyset$. In particular, applying  \cite[Cor.~4.14]{BalmerKrauseStevenson19} we have $\overbar{h}_{\cat B}(f_{ Y_{2}}) = 0$. Because $\overbar{h}_{\cat B}$ is a monoidal functor, we therefore have $\overbar{h}_{\cat B}(\kappaQ) = 0$ as well, as required.

	Finally, we prove that if $\cat Q \in \Supp_{\cat T}(E_{\cat B})$ then $\cat Q \in \gen(\cat P)$.  To this end, consider $\cat Q \notin \gen(\cat P)$, i.e., a prime ideal $\cat Q$ with $\cat P \nsubseteq \cat Q$.  Let $x \in \cat P \setminus \cat Q$ a compact object of~$\cat T$, so that $\cat P \notin \supp_{\cat T}(x)$ and $\cat Q \in \supp_{\cat T}(x)$.  Since $\phi^{-1}(\supp_{\cat T}(x)) = \Supph_{\cat T}(x)$ for compact objects, we deduce that $\cat B \notin \Supph_{\cat T}(x)$ and therefore $E_{\cat B} \otimes x = 0$.  (Here we use that the `naive' homological support and the homological support agree on compact objects \cite[Prop.~4.4]{Balmer20_bigsupport}).  Since $\Supp_{\cat T}$ satisfies the half-tensor product formula (see \cite[Thm.~7.22]{BalmerFavi11} for the noetherian case and \cite[Lem.~2.18]{bhs1} in general), we then have
	\[
		\emptyset = \Supp_{\cat T}(E_{\cat B} \otimes x) = \Supp_{\cat T}(E_{\cat B}) \cap \Supp_{\cat T}(x).
	\]
	Because $\cat Q \in \supp_{\cat T}(x) = \Supp_{\cat T}(x)$ by assumption, this gives $\cat Q \notin \Supp_{\cat T}(E_{\cat B})$, which finishes the proof.
\end{proof}

\begin{Lem}\label{lem:suppheY}
	For any Thomason subset $Y \subseteq \Spc(\cat T^c)$, we have
	\[
		\Supph_{\cat T}(e_Y) = \phi^{-1}(Y) \qquad\text{ and }\qquad \Supph_{\cat T}(f_Y) = \phi^{-1}(Y^c).
	\]
\end{Lem}

\begin{proof}
	Recall that $\Supph(x) = \phi^{-1}(\supp(x))$ for any compact object $x \in \cat T^c$ (see \cite[Prop.~4.4]{Balmer20_bigsupport}).  Also recall that $Y = \bigcup_{x \in \cat T^c_Y} \supp(x)$ where 
	\[
		\cat T^c_Y = \SET{x \in \cat T^c}{\supp(x)\subseteq Y}
	\]
	and $\Loco{e_Y} = \Loco{\cat T^c_Y}$. The formal properties of homological support (\cite[Prop.~4.3 and Thm.~4.5]{Balmer20_bigsupport}) then imply 
	\begin{align*}
		\Supph_{\cat T}(e_Y)
			= \Supph_{\cat T}(\Loco{e_Y}) 
		 	= \bigcup_{x \in \cat T^c_Y} \Supph_{\cat T}(x)
		 	= \phi^{-1}(Y).
	\end{align*}
	Moreover, since $\phi$ is surjective, $\Spch(\cat T^c) =\phi^{-1}(\supp(\unit)) = \Supph_{\cat T}(\unit)$.  Hence $\Spch(\cat T^c) = \Supph_{\cat T}(e_Y) \cup \Supph_{\cat T}(f_Y)$.  By the tensor-product theorem (\cref{rem:supph-tensor-product}), $\Supph_{\cat T}(e_Y) \cap \Supph_{\cat T}(f_Y) = \Supph_{\cat T}(e_Y \otimes f_Y) = \emptyset$ so that $\Supph_{\cat T}(f_Y) = \Supph_{\cat T}(e_Y)^c = \phi^{-1}(Y)^c = \phi^{-1}(Y^c)$.
\end{proof}

\begin{Cor}\label{cor:supphgamma}
	For any $\cat P \in \Spc(\cat T^c)$, we have $\phi^{-1}(\{\cat P\}) = \Supph_{\cat T}(\kappaP)$.
\end{Cor}

\begin{proof}
	If $\{\cat P\} = Y_1 \cap Y_2^c$ then $\kappaP = e_{Y_1} \otimes f_{Y_2}$ and hence 
	\[
		\Supph_{\cat T}(\kappaP) = \Supph_{\cat T}(e_{Y_1}) \cap \Supph_{\cat T}(f_{Y_2}) = \phi^{-1}(Y_1) \cap \phi^{-1}(Y_2^c) = \phi^{-1}(\{\cat P\})
	\]
	by the tensor-product theorem (\cref{rem:supph-tensor-product}) and \cref{lem:suppheY}.
\end{proof}

\begin{Prop}\label{prop:detection}
	For any $t \in \cat T$, we have 
	\begin{equation*}
		\phi(\Supph_{\cat T}(t)) \subseteq \Supp_{\cat T}(t).
	\end{equation*}
	In particular, if $\Supph_{\cat T}$ has the detection property then $\Supp_{\cat T}$ also has the detection property.
\end{Prop}

\begin{proof}
	For any $\cat B \in \Spch(\cat T^c)$, we have $\cat B \in \phi^{-1}(\{\phi(\cat B)\}) = \Supph_{\cat T}(\kappaphiB)$ by \cref{cor:supphgamma}.  In particular, if $\cat B \in \Supph_{\cat T}(t)$ then
	\[
		\cat B \in \Supph_{\cat T}(\kappaphiB) \cap \Supph_{\cat T}(t) = \Supph_{\cat T}(\kappaphiB \otimes t)
	\]
	by the tensor-product theorem (\cref{rem:supph-tensor-product}). In particular, $\kappaphiB \otimes t \neq 0$ so that $\phi(\cat B) \in \Supp(t)$. Finally, the established inclusion implies that if $\Supp_{\cat T}(t) = \emptyset$ for some $t \in \cat T$, then $\Supph_{\cat T}(t) = \emptyset$ as well.  The detection property for $\Supph_{\cat T}$ would then imply that $t=0$.
\end{proof} 

\begin{Rem}
	As demonstrated in \cref{ex:oddring} below, the inclusion established in \cref{prop:detection} is not always an equality. Moreover, the same example shows that the detection property for $\Supp_{\cat T}$ is not sufficient to guarantee the detection property for $\Supph_{\cat T}$.
\end{Rem}

\begin{Cor}\label{cor:supph1}
    The following conditions are equivalent:
    \begin{enumerate}
        \item The detection property holds for $\Supph_{\cat T}$.
        \item The detection property holds for $\Supp_{\cat T}$ and
			\[
				\phi(\Supph_{\cat T}(t)) = \Supp_{\cat T}(t)
			\]
			 for all $t \in \cat T$.
    \end{enumerate}
\end{Cor}

\begin{proof}
	(a)$\Rightarrow$(b): The detection property for $\Supp_{\cat T}$ follows from the detection property for $\Supph_{\cat T}$, as shown in \cref{prop:detection}.  Moreover, in order to establish the equality in~(b) we only need to verify the inclusion $\Supp_{\cat T}(t) \subseteq \phi(\Supph_{\cat T}(t))$, again by \cref{prop:detection}.  To this end, let $\cat P \in \Supp_{\cat T}(t)$, so that $\kappaP \otimes t \neq 0$.  By the assumed detection property for $\Supph_{\cat T}$ we then have
	\[
		\emptyset \neq \Supph_{\cat T}(\kappaP \otimes t) = \Supph_{\cat T}(\kappaP) \cap \Supph_{\cat T}(t) = \phi^{-1}(\{\cat P\}) \cap \Supph_{\cat T}(t)
	\]
	by the tensor product theorem (\cref{rem:supph-tensor-product}) and \cref{cor:supphgamma}. In other words, we have $\cat P \in \phi(\Supph_{\cat T}(t))$, as desired. 

	(b)$\Rightarrow$(a): If $\Supph_{\cat T}(t) = \emptyset$ then $\Supp_{\cat T}(t) = \phi(\Supph_{\cat T}(t)) = \emptyset$ as well, hence $t=0$ by the detection property for $\Supp_{\cat T}$.
\end{proof}

\begin{Prop}\label{prop:tpf-implies-balmerconj}
	If $\Supp_{\cat T}$ satisfies the tensor product formula then the comparison map $\phi\colon \Spch(\cat T^c) \to \Spc(\cat T^c)$ is a bijection. If $\Supph_{\cat T}$ has the detection property then the converse holds.
\end{Prop}

\begin{proof}
	($\Rightarrow$) The map $\phi$ is always surjective (\cref{rem:homological_spectrum_map}) so it suffices to show it is injective. Assume by way of contradiction that we have maximal Serre $\otimes$-ideals $\cat B,\cat B'$ with $\cat B \ne \cat B'$ and $\phi(\cat B) = \phi(\cat B')$. By \cite[Prop.~5.3]{Balmer20_nilpotence} we must have $E_{\cat B} \otimes E_{\cat B'} = 0$. By \cref{lem:support_injectives} and the assumed tensor product formula this gives 
    \[
		\emptyset = \Supp_{\cat T}(E_{\cat B} \otimes E_{\cat B'}) = \Supp_{\cat T}(E_{\cat B}) \cap \Supp_{\cat T}(E_{\cat B'}) \ni \phi(\cat B) = \phi(\cat B')
    \]
	which is absurd.

	We will now prove the converse assuming that $\Supph_{\cat T}$ has the detection property. Since the inclusion $\Supp_{\cat T}(x\otimes y) \subseteq \Supp_{\cat T}(x) \cap \Supp_{\cat T}(y)$ holds for any $x, y \in \cat T$, it suffices to check the reverse inclusion. To this end, let $\cat P \in \Supp_{\cat T}(x) \cap \Supp_{\cat T}(y)$ and write $\cat B \coloneqq \phi^{-1}(\cat P)$. In particular, we have $x \otimes \kappaP \neq 0$, so combining the detection property for $\Supph_{\cat T}$ with \cref{prop:detection} we get
	\[
		\emptyset \neq \phi(\Supph_{\cat T}(x \otimes \kappaP)) \subseteq \Supp_{\cat T}(x \otimes \kappaP)) \subseteq \{\cat P\}
	\]
	and hence $\phi(\Supph_{\cat T}(x \otimes \kappaP)) = \{\cat P\}$. The same argument also works for $y$. Because~$\kappaP$ is idempotent and $\Supph_{\cat T}$ satisfies the tensor product formula, this implies
	\[
		\Supph_{\cat T}(x \otimes y \otimes \kappaP) = \Supph_{\cat T}(x \otimes \kappaP) \cap \Supph_{\cat T}(y \otimes \kappaP) = \{\cat B\}.
	\]
	In particular, $x \otimes y \otimes \kappaP \neq 0$ and therefore $\cat P \in \Supp_{\cat T}(x\otimes y)$, as desired.
\end{proof} 

\begin{Rem}
	The tensor product formula is not a necessary condition for $\phi$ to be bijective. For example, the comparison map is bijective for the derived category of any commutative ring \cite[Cor.~5.11]{Balmer20_nilpotence}, but there are examples of commutative rings for which the tensor product formula does not hold; see \cref{ex:oddring}.
\end{Rem}

\begin{Prop}\label{prop:supph3}
	The following two conditions are equivalent:
    \begin{enumerate}
		\item The comparison map is bijective and $\phi(\Supph_{\cat T}(t)) = \Supp_{\cat T}(t)$ for all $t \in \cat T$.
        \item $\Supph_{\cat T}(t) = \phi^{-1}(\Supp_{\cat T}(t))$ for all $t \in \cat T$.
    \end{enumerate}
	Both conditions imply:
    \begin{enumerate}
        \item[(c)] $\Supp_{\cat T}$ satisfies the tensor product formula.
    \end{enumerate}
	The converse holds, that is, (c) implies (a) and (b), if $\Supph_{\cat T}$ has the detection property.
\end{Prop}
\begin{proof}
	(a)$\Rightarrow$(b): Applying $\phi^{-1}$ to the formula in (a) yields the claim.

	(b)$\Rightarrow$(a): The comparison map is surjective, hence 
	\[
		\phi(\Supph_{\cat T}(t)) = \phi(\phi^{-1}(\Supp_{\cat T}(t))) = \Supp_{\cat T}(t)
	\]
	for all $t \in \cat T$. It thus remains to show that $\phi$ is also injective. To this end, let $\cat B \in \Spch(\cat T^c)$ and compute using \cref{lem:support_injectives}:
	\[
		\phi^{-1}(\{\phi(\cat B)\}) = \phi^{-1}(\Supp_{\cat T}(E_{\cat B})) = \Supph_{\cat T}(E_{\cat B}) = \{\cat B\}.
	\]
	Hence $\phi$ is injective.

	(a)$\Rightarrow$(c): This follows from the tensor product formula for $\Supph_{\cat T}$, which was established in \cite[Thm.~1.2(d)]{Balmer20_bigsupport}. 

	(c)$\Rightarrow$(a): Assume now that $\Supph_{\cat T}$ detects trivial objects. The comparison map is a bijection by \cref{prop:tpf-implies-balmerconj}, so Statement (a) is a consequence of \cref{cor:supph1} and \cref{prop:tpf-implies-balmerconj}. 
\end{proof} 

We can now assemble the proof of \cref{thm:B} stated in the Introduction.

\begin{proof}[Proof of \cref{thm:B}]
	Statement (a) is the content of \cref{prop:detection}. The implications $(1) \Rightarrow (2)$ and $(2) \Rightarrow (3)$ are part of \cref{prop:supph3} and \cref{prop:tpf-implies-balmerconj}, which also establish the respective converses assuming the detection property for $\Supph_{\cat T}$.
\end{proof}

\section{The main theorem}

In this section we prove our main theorem (\cref{thm:A} in the Introduction), which establishes that the homological spectrum and Balmer spectrum coincide, and that the homological support and Balmer--Favi support coincide, when the category is stratified. 
We refer the reader to \cite[Definition~4.4]{bhs1} for our terminology regarding stratification;
in a nutshell, a category is stratified if the Balmer--Favi notion of support provides a correspondence between the localizing tensor-ideals and the subsets of the Balmer spectrum.
Along the way, we establish a point-set topological criterion on $\Spch(\cat T^c)$ for the comparison map $\phi$ to be a homeomorphism.

\begin{Rem}
	Recall that the Kolmogorov quotient $\mathrm{KQ}(X)$ of a topological space~$X$ is its reflection into the category of $T_0$-spaces (a.k.a.~Kolmogorov spaces).  In other words,	$\mathrm{KQ}(X)$ is the initial $T_0$-space equipped with a continuous map from $X$. It can be constructed explicitly as follows. Two points $x,y \in X$ are said to be \emph{topologically indistinguishable} if $\overbar{\{x\}}=\overbar{\{y\}}$.  We denote the resulting equivalence relation on~$X$ by $\equiv$. The Komolgorov quotient $\mathrm{KQ}(X)$ is then the quotient space of $X$ under $\equiv$:
    \[
		X \to \mathrm{KQ}(X) := X/\equiv.
    \]
    In particular, if $X$ is already $T_0$, then the quotient map $X \to \mathrm{KQ}(X)$ is a homeomorphism. 
\end{Rem}

\begin{Lem}\label{lem:kq}
    The comparison map $\phi$ of \cref{rem:homological_spectrum_map} exhibits $\Spc(\cat T^c)$ as the Kolmogorov quotient of $\Spch(\cat T^c)$. In particular, for any $\cat B_1,\cat B_2 \in \Spch(\cat T^c)$ we have 
    \begin{equation}\label{eq:image-char}
		\phi(\cat B_1) = \phi(\cat B_2) \Longleftrightarrow \overbar{\{\cat B_1\}} = \overbar{\{\cat B_2\}}.
	\end{equation}
\end{Lem}

\begin{proof}
	We first verify \eqref{eq:image-char}, which identifies the equivalence relation $\equiv$ of topological indistinguishability on $\Spch(\cat T^c)$ with the equivalence relation induced by the function $\phi$. To this end, we observe that 
	\begin{equation}\label{eq:Bclosure}
		\overbar{\{ \cat B\}} = \phi^{-1}(\overbar{\{\phi(\cat B)\}})
	\end{equation}
	for any $\cat B \in \Spch(\cat T^c)$.  This is a routine verification from the definitions.  Indeed, the topology on $\Spch(\cat T^c)$ is defined by taking a basis of closed sets to be those sets of the form $\supph(x) = \phi^{-1}(\supp(x))$ for $x \in \cat T^c$.  Hence
	\[
		\begin{split}
			\overbar{\{ \cat B\}} & = \bigcap_{\substack{x \in \cat T^c \,:\\ \cat B \in \supph(x)}} \supph(x)
			= \bigcap_{\substack{x \in \cat T^c \,:\\ \phi( \cat B ) \in \supp(x)}} \phi^{-1}(\supp(x))  
			=\phi^{-1}(\overbar{\{\phi(\cat B)\}}).
		\end{split}
	\]
	The ($\Rightarrow$) direction of \eqref{eq:image-char} is then immediate from \eqref{eq:Bclosure}.  For the converse, observe that if $x \not\in \phi(\cat B_1)$ then $\phi(\cat B_1) \in \supp(x)$ so that $\overbar{\{\phi(\cat B_1)\}} \subseteq \supp(x)$.  Thus, if $\overbar{\{\cat B_1\}} = \overbar{\{\cat B_2\}}$, we have
	\[
		\phi^{-1}(\overbar{\{\phi(\cat B_2)\}}) = \phi^{-1}(\overbar{\{\phi(\cat B_1)\}}) \subseteq \phi^{-1}(\supp(x))
	\]
	by \eqref{eq:Bclosure}, and it follows that $x \not\in \phi(\cat B_2)$. This establishes the $(\Leftarrow)$ direction of \eqref{eq:image-char}.
	
	Since $\phi$ determines the same equivalence relation on $\Spch(\cat T^c)$ as the Kolmogorov quotient, it remains to prove that $\phi$ is a quotient map.  It is continuous by construction and is surjective by \cite[Cor.~3.9]{Balmer20_nilpotence}.  All that remains is to show that $\Spc(\cat T^c)$ has the finest topology for which $\phi$ is continuous.  As noted above, every basic closed set of $\Spch(\cat T^c)$ is the preimage of a basic closed set of $\Spc(\cat T^c)$. Hence if $Y \subseteq \Spc(\cat T^c)$ is a subset such that $\phi^{-1}(Y)$ is closed, then we can write
	\[
		\phi^{-1}(Y) = \bigcap \phi^{-1}(A_i) = \phi^{-1}(\bigcap A_i)
	\]
	for basic closed sets $A_i \subseteq \Spc(\cat T^c)$. Since $\phi$ is surjective, this implies $Y = \bigcap A_i$ is closed, as desired.
\end{proof}

\begin{Prop}\label{prop:homeo}
	The following are equivalent:
	\begin{enumerate}
		\item The comparison map $\phi$ is a bijection.
		\item The comparison map $\phi$ is a homeomorphism.
		\item The homological spectrum $\Spch(\cat T^c)$ is a spectral space.
		\item \label{it:Spch-T0} The homological spectrum $\Spch(\cat T^c)$ is $T_0$.
	\end{enumerate}
\end{Prop}

\begin{proof}
	The topology on $\Spch(\cat T^c)$ is defined by taking a basis of closed sets to be those sets of the form $\Supph_{\cat T}(x) = \phi^{-1}(\supp(x))$ for $x \in \cat T^c$.  That is, we pull back the usual basis of the topology on the Balmer spectrum.  From this observation, (a)$\Rightarrow$(b) is immediate.  Moreover, we have (b)$\Rightarrow$(c)$\Rightarrow$(d) simply because the Balmer spectrum $\Spc(\cat T^c)$ is a spectral space.  Finally, (d)$\Rightarrow$(a) follows from \cref{lem:kq}, so the proof is complete. 
\end{proof}

\begin{Rem}
	\cref{lem:kq} and \cref{prop:homeo} have nothing to do with the big tt-category $\cat T$. They hold with $\cat T^c$ replaced by any essentially small rigid tt-category~$\cat K$.  The rigid hypothesis ensures that the comparison map $\phi:\Spch(\cat K) \to \Spc(\cat K)$ is surjective (as proved in \cite[Cor.~3.9]{Balmer20_nilpotence}).
\end{Rem}

\begin{Thm}\label{thm:homological}
	Let $\cat T$ be a rigidly-compactly generated tt-category with $\Spc(\cat T^c)$ weakly noetherian (\cref{rem:weakly-noetherian}).  If $\cat T$ is stratified, then $\cat T$ satisfies the equivalent statements of \cref{prop:supph3}. In particular, the comparison map 
	\[
		\phi\colon \Spch(\cat T^c) \to \Spc(\cat T^c)
	\]
	is a bijection (hence a homeomorphism), $\Supph_{\cat T}(t) = \phi^{-1}(\Supp_{\cat T}(t))$ for all ${t \in \cat T}$, and both $\Supp$ and $\Supph$ have the tensor product property and the detection property (\cref{ter:suppproperties}).
\end{Thm}

\begin{proof}
	For any $\cat B \in \Supph_{\cat T}$ we have $\Supp_{\cat T}(E_{\cat B}) \simeq \{\phi(\cat B)\}$ by \Cref{lem:support_injectives}. If we can show that $\Supph_{\cat T}$ has the detection property, then the claim will follow from \cref{prop:supph3}, \cref{prop:homeo} and \cref{rem:supph-tensor-product}.  To this end, let $t \in \cat T$ be a non-zero object. Since $\cat T$ is stratified, $\Supp_{\cat T}$ detects trivial objects, so there exists $\cat P \in \Spc(\cat T^c)$ with $\kappaP \otimes t \neq 0$. Since $\phi$ is surjective, we can choose some $\cat B \in \phi^{-1}(\{\cat P\})$. Minimality at $\cat P$ and $\{\cat P\} = \Supp_{\cat T}(E_{\cat B})$ implies that $E_{\cat B}\in \Loco{\kappaP \otimes t}$. If $\cat B \notin \Supph_{\cat T}(\kappaP \otimes t)$ then $[\kappaP \otimes t,E_{\cat B}] = 0$ and hence $[E_{\cat B},E_{\cat B}] = 0$, a contradiction. Therefore, we have $\cat B \in \Supph_{\cat T}(\kappaP \otimes t)$. This implies
    \[
		\cat B \in \Supph_{\cat T}(t \otimes \kappaP) = \Supph_{\cat T}(t) \cap \Supph_{\cat T}(\kappaP).
    \]
	In particular, we have $\cat B \in \Supph_{\cat T}(t)$, as desired.
\end{proof}

\begin{Rem}\label{rem:naive_support}
	The proof of \cref{thm:homological} also shows that, assuming stratification, homological support coincides with the `naive' notion of homological support (see \cref{rem:naivesupph}).
\end{Rem}

\begin{Rem}
	As explained in \cite[Section 7]{bhs1}, the Balmer--Favi notion of support is --- in weakly noetherian contexts --- the universal notion of support for the purposes of stratification.  For example, if $\cat T$ is stratified in the sense of Benson--Iyengar--Krause by the action of a graded noetherian ring (\cref{ex:bik}) then $\cat T$ is stratified in our sense (that is, by the Balmer--Favi notion of support).  Moreover, the original notion of support can be identified with the Balmer--Favi notion of support.  This invocation of \cite[Theorem~7.6]{bhs1} and its corollaries will be used repeatedly without further comment in the examples below.
\end{Rem}

\section{Applications and examples}\label{sec:examples}

We now turn to applications of the above results concerning the relationship between the Balmer--Favi support and the homological support.

\subsection*{Examples from commutative algebra}

\begin{Rem}\label{rem:DR-example}
	If $\cat T = \Der(R)$ is the derived category of a commutative ring then $\Spch(\cat T^c) \cong \Spc(\cat T^c)$ via the comparison map $\phi$ (see \cite[Corollary 5.11]{Balmer20_nilpotence}), which in turn is homeomorphic to $\Spec(R)$ by Thomason's theorem \cite{Thomason97}. Under these identifications, we claim that the homological support of an object $X \in \Der(R)$ is given by 
	\[
		\Supph(X) = \{\mathfrak{p} \in \Spec(R)\mid X \otimes \kappa(\mathfrak{p}) \neq 0\}
	\]
	where $\kappa(\mathfrak{p})$ denotes the residue field of $R$ at the prime ideal $\mathfrak{p}$. Indeed, Balmer and Cameron \cite{BalmerCameron20pp} prove that $E_{\cat B} \simeq \kappa(\mathfrak{p})$ for all $\cat B \in \Spch(\Der(R)^c)$ where $\mathfrak{p}$ is the prime ideal corresponding to $\cat B$. Because $\Hom_R(M,\kappa(\mathfrak{p}))\simeq \Hom_{\kappa(\mathfrak{p})}(M\otimes\kappa(\mathfrak{p}),\kappa(\mathfrak{p}))$ vanishes if and only if $M\otimes\kappa(\mathfrak{p}) \simeq 0$, the claim follows.

	For comparison, the Balmer--Favi support can be given in terms of Koszul complexes. Indeed, for a finite sequence $\underline{x} = x_1,\ldots,x_r \in R$, we define 
	\[
		K_{\infty}(\underline{x}) \coloneqq (R \to R[1/{x_1}]) \otimes \cdots \otimes (R \to R[1/{x_r}])
	\]
	and more generally $K_{\infty}(\underline{x}; M) \coloneqq K_{\infty}(\underline{x}) \otimes M$. Then
	\[
		\begin{split}
			\Supp(X) = \{ \frak p \in \Spec(R) \mid K_{\infty}(\underline{x};X_{\frak p}) \ne 0 \text{ for every finite sequence } \underline{x} \in \frak p \},
		\end{split}
	\]
	see \cite[Lemma~5.2]{BillySanders2017pp}. If $\frak p = (x_1,\ldots,x_n)$ is finitely generated (for example, if $R$ is noetherian), then it suffices to check this for the sequence $(x_1,\ldots,x_n)$. 
\end{Rem}

\begin{Exa}\label{ex:noetherian_ring}
	If $R$ is noetherian, then $\Der(R)$ is stratified (as established by Neeman \cite{Neeman92a}; see \cite[Example~5.7]{bhs1}). Hence \cref{thm:homological} applies and we conclude that $\Supph(X)$ and $\Supp(X)$ agree for all $X \in \Der(R)$. This can also be deduced from work of Foxby and Iyengar \cite[Thm.~2.1 and Thm.~4.1]{FoxbyIyengar03}; see also \cite[Rem.~9.2]{BensonIyengarKrause08}. 
\end{Exa}

\begin{Exa}
	The previous example can be extended in a number of ways. For example, let $R$ be a $G$-graded ring, where $G$ is an abelian group. Then if $R$ is an \mbox{$\epsilon$-commutative} noetherian ring (see \cite[Def.~2.4]{DellAmbrogioStevenson13}), the derived category of graded $R$-modules $\Der(R)$ is stratified by $\Spc(\Der(R)^c) \cong \Spec(R)$ (see \cite[Thm.~5.7]{DellAmbrogioStevenson13}).  In particular, the homological spectrum and the Balmer spectrum can both be identified with $\Spec(R)$. The argument in the non-graded case (\cite[Cor.~3.3]{BalmerCameron20pp}) goes through verbatim to identify the pure-injective $E_{\cat{B}}$ corresponding to the prime~$\frak p$ with the residue field $\kappa(\frak p) \coloneqq R_{\frak p}/\frak pR_{\frak p}$ (which is a field in the graded sense; see \cite[Lem.~4.2]{DellAmbrogioStevenson13}). By \Cref{thm:homological}, the Balmer--Favi support agrees with the homological support 
	\[
		\Supph(X) = \{\mathfrak{p} \in \Spec(R)\mid X \otimes \kappa(\mathfrak{p}) \neq 0\}.
	\]
	This has also been proved directly by Dell'Ambrogio and Stevenson \cite[Cor.~5.6]{DellAmbrogioStevenson13}.
\end{Exa}

\begin{Exa}
	In another direction, one can instead assume that $R$ is a commutative dg-algebra with $H^*(R)$ noetherian, such that $R$ is formal.  In this case, $\Der(R)$ is stratified by $\Spc(\Der(R)^c) \cong \Spec(H^*(R))$ (see \cite[Thm.~8.1]{BensonIyengarKrause11b}). Similarly to \Cref{ex:noetherian_ring}, one identifies the Balmer--Favi support with that defined by Koszul complexes and the homological support with that defined via residue fields.  \Cref{thm:homological} implies that these two approaches define equivalent theories of support.  We leave the details to the interested reader. 
\end{Exa}

\begin{Exa}\label{ex:oddring}
	Let $k$ be a field, and let $R$ be the ring 
	\[
		R = \frac{k[x_2,x_3,\ldots]}{(x_2^2,x_3^3,\ldots)}
	\]
	considered by Neeman in \cite{Neeman00}. Even though $R$ is non-noetherian, $\Spec(R) = \Spc(\Der(R)^c)$ is a point and so is a noetherian space.  However, Neeman shows that~$\Der(R)$ has many localizing tensor ideals and so $\Der(R)$ cannot be stratified. To see this from the perspective of \Cref{thm:homological}, we claim that there exists a non-zero complex $I \in \Der(R)$ such that $\Supph(I) = \emptyset$, so that $\Supph$ cannot have the detection property.  Indeed, there exists a non-zero complex $I \in \Der(R)$ with $I \otimes I = 0$ (see \cite[Thm.~C]{DwyerPalmieri08}).  Because $\Supph$ has the tensor product property, this means it cannot have the detection property (see \Cref{rem:tensor-nilpotence}).  
	Note that in this case $\Supp$ detects the triviality of a complex (see \cite[Example 4.6]{bhs1}), so there cannot be a tensor product formula for $\Supp$. Consequently, the Balmer--Favi support and the homological support do not agree in this example.
\end{Exa}

\begin{Rem}
	The above example shows that the homological support and the Balmer--Favi support can differ even if the homological and tensor triangular spectra coincide.
\end{Rem}

\begin{Exa}
    Let $R$ be the ring denoted $A$ by Keller in \cite[Section~2]{Keller94b}.
	It is a non-discrete valuation domain of rank 1 whose value group is $\bbZ[1/\ell] \subset \bbQ$; see \cite[Theorem II.3.8]{FuchsSalce01} and
	\cite[Example~5.24]{BazzoniStovicek17}.
	In particular, $\Spec(R) =\{0,\frak m\}$.
	The residue fields for the two prime ideals are
	$\kappa(0)=Q$
	and $\kappa(\frak m)=R/\frak m$, respectively.
	Inspired by \cite[Example~5.7]{BillySanders2017pp}, consider
$X \coloneqq Q/R \otimes_R^{\bbL} \frak m \in \Der(R)$.
	Since $\Tor_1^R(Q/R,R/\frak m) = R/\frak m \neq 0$,
	the exact sequence 
	\[0\to \Tor_1^R(Q/R,R/\frak m)\to Q/R \otimes_R \frak m\]
	implies that $Q/R \otimes_R \frak m \neq 0$.
	Hence $X\in \Der(R)$ is nonzero.
	On the other hand,
	$X \otimes_R^{\bbL} Q = 0$
	since $Q/R \otimes_R^{\bbL} Q=0$.
	Moreover, the ideal $\frak m$ is 
	flat
	(since valuation domains have weak dimension at most one \cite[Theorem~10.4]{FuchsSalce01})
	and the nature of the value group $\bbZ[1/\ell]$ implies that $\frak m^2=\frak m$.
	It follows that $\frak m \otimes_R^{\bbL} R/\frak m= 0$ in $\Der(R)$
	and hence 
	$X \otimes_R^{\bbL} R/\frak m =0$.
	This establishes, by \cref{rem:DR-example},
	that $\Supph(X) = \emptyset$.
	Thus $\Supph$ does not have the detection property.
	By \cref{thm:homological}, $\Der(R)$ cannot be stratified.
\end{Exa}

\subsection*{Stable homotopy theory.}\hspace{1em}

\smallskip

\noindent
Let $\Sp$ denote the stable homotopy category of $p$-local spectra for a fixed prime $p$. We recall the description of $\Spc(\Sp^c)$ due to Hopkins and Smith \cite{HopkinsSmith98}. 

\begin{Not}
	For each $0 \le h \le \infty$, let $\cat C_h \coloneqq \SET{ x \in \Sp^c }{K(h)_*(x) = 0}$.
\end{Not}

\begin{Rem}
	Here $K(h)_*$ is the Morava $K$-theory homology:
	\[
		K(h)_* \colon \Sp \to  \bbF_{p}[v_h^{\pm 1}]\GrMMod
	\]
	with $|v_h| = 2(p^h-1)$, and we make the convention that the target is graded \mbox{$\bbQ$-modules} when $h = 0$ and graded $\Fp$-modules when $h = \infty$; i.e., we take $K(0) = \HQ$ and $K(\infty) = \HFp$. As interpreted by Balmer \cite[Cor.~9.5]{Balmer10b}, the $\cat C_h$ are exactly the prime ideals of $\Sp^c$. 
\end{Rem}

\begin{Thm}[Hopkins--Smith]\label{thm:hopkins_smith_thick}
	The spectrum
	\[
		\Spc(\cat \Sp^c) = \cat C_\infty - \dots - \cat C_{h+1} - \cat C_{h} - \cdots - \cat C_{1} - \cat C_0
	\]
	is an infinite tower of connected points, where closure goes to the left: $\overbar{\{\cat C_h\}} = \SET{\cat C_k}{h \le k \le \infty}$.  In particular, the space is irreducible with generic point $\cat C_0 = \Sp_{\mathrm{tor}}^c$ and $\cat C_\infty = (0)$ is the unique closed point.
\end{Thm}

\begin{Rem}
	By \cite[Cor.~5.10]{Balmer20_nilpotence}, the comparison map 
	\[
		\phi \colon \Spch(\Sp^c) \to \Spc(\Sp^c)
	\]
	is a bijection. Moreover, if $\cat B \in \Spch(\Sp^c)$ is the homological prime corresponding to $\cat C_h$, we have an isomorphism $E_{\cat B} \simeq K(h)$, see \cite[Cor.~3.6]{BalmerCameron20pp}.  In particular, the homological support is given by 
	\[
		\Supph(x) = \SET{ h \in \mathbb{N} \cup \infty}{ [x,K(h)] \neq 0 }. 
	\]
	Because $[x,K(h)] \simeq [K(h) \otimes x, K(h)]_{\Mod_{K(h)}}$ vanishes if and only if $ K(h) \otimes x = 0$ since $K(h)_*$ is a graded field, we deduce that 
	\begin{equation}\label{eq:Supph-Sp}
		\Supph(x) = \SET{ h \in \mathbb{N} \cup \infty}{ K(h) \otimes x \neq 0 }.
	\end{equation}
\end{Rem}

\begin{Rem}
	The space $\Spc(\Sp^{c})$ is not weakly noetherian as the closed point~$\cat C_{\infty}$ is not weakly visible. This is because $\{\cat C_{\infty}\}$ is not a \emph{Thomason} closed subset. For $0 \le h < \infty$, we claim that the idempotent $\kappaCh$ is isomorphic to $M_h^fS^0$, the fiber of the natural morphism $L^f_hS^0 \to L_{h-1}^fS^0$. Indeed $g(\cat C_h) \simeq e_{\overbar{\{\cat C_h\}}} \otimes f_{Y_{\cat C_h}}$ where $e$ and~$f$ denote the left and right tensor idempotents of the associated finite localizations (see \Cref{rem:idempotents}) and $Y_{\cat C_h} = \supp(\cat C_h)$.  It follows from the definitions that $f_{Y_{\cat C_h}} \simeq L^f_{h}S^0$ (compare \cite[Example~5.12]{BalmerSanders17}, although note that our indexing differs by one) and $e_{\overbar{\{\cat C_h\}}} \simeq C_{h-1}^fS^0$, the fiber of the finite localization $S^0 \to L_{h-1}^fS^0$.  Therefore, $g(\cat C_h) \simeq C_{h-1}^fS^0 \otimes L_{h}S^0 \simeq C_{h-1}^fL_{h}S^0 \simeq M_h^fS^0$.  We deduce that 
	\[
		\Supp(x) = \SET{ h \in \mathbb{N} }{ M^f_hS^0 \otimes x \neq 0 }.
	\]
	An almost identical argument to \cite[Prop.~5.3]{HoveyStrickland99} shows that 
	\[
		M^f_hS^0 \otimes x \neq 0 \iff T(h) \otimes x \neq 0
	\]
	where $T(h)$ is the telescope of a finite type $h$ spectrum. Hence
	\[
		\Supp(x) = \SET{  h \in \mathbb{N} }{ T(h) \otimes x \neq 0 }.
	\]
\end{Rem}

\begin{Rem}
	The homological and triangular support as defined above can never agree, since $\Supph(\HFp) = \{\infty\}$ and yet the point $\infty$ is not seen by $\Supp$ since the corresponding point $\cat C_{\infty} \in \Spc(\Sp^c)$ is not weakly visible.  However, if we set $T(\infty) \coloneqq \HFp$, we can define an extended theory of triangulated support by
	\begin{equation}\label{eq:Suppinfty-Sp}
		\Supp_{\le \infty}(x) \coloneqq \SET{  h \in \mathbb{N} \cup \infty }{ T(h) \otimes x \neq 0 }. 
	\end{equation}
	This defines a support function for $\Sp$ taking values in $\mathbb{N} \cup \infty \cong \Spc(\Sp^c)$.
	It 
	simply
	completes the Balmer--Favi support $\Supp(x)$ by possibly including the point at infinity, i.e., the closed point.  Comparing the homological support \eqref{eq:Supph-Sp} and extended triangular support \eqref{eq:Suppinfty-Sp}, we see that they agree if and only if 
	\[
		T(h) \otimes x \neq 0 \iff K(h) \otimes x \neq 0
	\]
	for $0 \le h < \infty$. This is precisely the telescope conjecture; see \cite[1.13]{MahowaldRavenelShick01}. We thus obtain:
\end{Rem}

\begin{Prop}
	The extended triangular support and the homological support on the $p$-local stable homotopy category $\Sp$ agree if and only if the telescope conjecture holds.
\end{Prop}

\begin{Rem}
	For any $x \in \Sp$, we have $T(h) \otimes x = 0 \implies K(h) \otimes x = 0 $, so that $\Supph(x) \subseteq \Supp_{\le \infty}(x)$ always holds. However, neither support function has the detection property: the Brown--Comenetz dual of the sphere is a counter-example. Indeed, for the case of $T(h)$, see \cite[Lemma 7.1(d)]{HoveyPalmieri99}, while the cases of $K(h), \HQ$ and $\HFp$ are given by \cite[Corollary B.12]{HoveyStrickland99}.
\end{Rem}

\begin{Exa}
	Suppose we localize and work instead with the category of $E(n)$-local spectra $\cat T=\Sp_{E(n)}$.  The telescope conjecture holds in this category by \cite[Corollary~6.10]{HoveyStrickland99} and an analysis similar to the above shows that
	\[
		\Supph(x) = \SET{ h \in \{0,\ldots n\} }{ K(h) \otimes x \neq 0 }
	\]
	while
	\[
		\Supp(x) = \SET{ h \in \{0,\ldots n\} }{ M_hS^0 \otimes x \neq 0 }.
	\]
	These agree by \cite[Proposition~5.3]{HoveyStrickland99}. Alternatively, this follows from \Cref{thm:homological}, as \cite[Theorem 10.14]{bhs1} establishes that $\Sp_{E(n)}$ is stratified.
\end{Exa}

\subsection*{Affine weakly regular tensor triangulated categories}

\begin{Def}[Dell'Ambrogio--Stanley \cite{DellAmbrogioStanley16}]
    A tensor triangulated category $\cat T$ is said to be \emph{affine weakly regular} if it satisfies the following two conditions:
    \begin{enumerate}
        \item (affine) $\cat T$ is compactly generated by its tensor unit $\unit$. 
        \item (weakly regular) The graded endomorphism ring $R \coloneqq \Hom_{\cat T}^*(\unit,\unit)$ is a graded noetherian ring concentrated in even degrees, and for every homogeneous prime ideal $\frak p$ of $R$, the maximal ideal of the local ring $R_{\frak p}$ is generated by a (finite) regular sequence of homogeneous non-zero divisors.  
    \end{enumerate}
\end{Def}

\begin{Rem}
	The first axiom ensures that $\cat T$ is a rigidly-compactly generated tt-category.
\end{Rem}

\begin{Rem}
	Given an affine weakly regular tt-category and a prime $\frak p \in \Spec(R)$, there is a residue field object $K(\frak p)$ with the property that 
	\[
		\pi_*(K(\frak p)) \coloneqq \pi_*\Hom_{\cat T}(\unit,K(\frak p)) \cong \kappa(\frak p)
	\]
	where $\kappa(\frak p) \coloneqq R_{\frak p}/\frak pR_{\frak p}$ denotes the algebraic residue field. See \cite[Section~3]{DellAmbrogioStanley16}.
\end{Rem}

\begin{Thm}[Dell'Ambrogio--Stanley {\cite[Theorem 1.3]{DellAmbrogioStanley16}}]\label{thm:weak_affine_strat}
    If $\cat T$ is an affine weakly regular tt-category, then $\cat T$ is stratified by $\Spc(\cat T^c) \cong \Spec(R)$. 
\end{Thm}

\begin{Thm}\label{thm:affine_comparison}
	Let $\cat T$ be an affine weakly regular tt-category. Then:
    \begin{enumerate}
        \item The comparison map  $\phi\colon \Spch(\cat T^c) \to \Spc(\cat T^c) \cong \Spec(R)$ is a homeomorphism and $\Supph(t) = \phi^{-1}(\Supp(t))$ for all $t \in \cat T$.
        \item For the homological prime $\cat B \in \Spch(\cat T^c)$ corresponding to $\frak p \in \Spec(R)$, we have an isomorphism $E_{\cat B} \simeq K(\frak p)$. 
    \end{enumerate}
\end{Thm}

\begin{proof}
	Part (a) is an immediate consequence of  \Cref{thm:weak_affine_strat} and \Cref{thm:homological}.  For part (b), we apply \cite[Lemma~2.2 and Theorem~3.1]{BalmerCameron20pp} to $F \colon \cat T \to \Mod_{K(\frak p)}$.  This implies that the corresponding pure-injective object $E_{\cat B}$ is a summand of $K(\frak p)$, but $K(\frak p)$ is indecomposable by \cite[Lemma~3.2]{BalmerCameron20pp}.
\end{proof}

\begin{Rem}
    One can also prove a nilpotence theorem for affine weakly regular tt-categories using the residue fields $K(\frak p)$.  The bijectivity of the comparison map $\phi$ can therefore also be proved using \cite[Theorem~5.4]{Balmer20_nilpotence}. Alternatively, by \cite[Theorem~1.1]{Balmer20_nilpotence} and \Cref{thm:affine_comparison} there is a nilpotence theorem as follows: 
\end{Rem}

\begin{Cor}
    Let $\cat T$ be an affine weakly regular tt-category and for $\frak p \in \Spec(R)$ write 
    \[
		K(\frak p)_*(x) \coloneqq \pi_*(K(\frak p) \otimes x).
	\]
    If $f \colon x \to y$ a morphism in $\cat T^c$ such that $K(\frak p)_*(f) = 0$ for all $\frak p \in \Spec(R)$, then there exists $n \gg 0$ such that $f^{\otimes n} = 0$ in $\cat T$. 
\end{Cor}

\begin{Rem}
	This generalizes the nilpotence theorem of \cite[Corollary~2.10]{Mathew15}.
\end{Rem}

\subsection*{Further examples}

\begin{Exa}[Cochain algebras]
	Let $X$ be a connected space and let $C^*(X;\Fp) \coloneqq F(\Sigma^{\infty}_+X,\HFp)$ denote the ring spectrum of $\Fp$-valued cochains on $X$. The question of when the homotopy category of $\Mod_{C^*(X;\Fp)}$ is stratified by $H^*(X;\Fp)$ has been investigated in \cite{BarthelCastellanaHeardValenzuela19,BarthelCastellanaHeardValenzuela19pp}. For example, this holds in the following cases:
	\begin{enumerate}
		\item $X=BG$ is the classifying space of a compact Lie group, a $p$-local group, a Kac--Moody group, or a connected $p$-compact group. 
		\item $X$ is a connected $H$-space with noetherian mod $p$ cohomology. 
	\end{enumerate}
	\Cref{thm:homological} then applies to show that the comparison map $\phi$ is a bijection. By Balmer's abstract nilpotence theorem \cite[Thm.~1.1]{Balmer20_nilpotence} this implies that for each $\cat P \in \Spc(\Mod_{C^*(X;\Fp)}^{c}) \cong \Spec(H^*(X;\Fp))$ there exists a unique homological tensor functor 
	\[
		\overbar{h}_{\cat B(\cat P)} \colon \Mod_{C^*(X;\Fp)} \to \cat {A}_{\cat P}
	\]
	to some Grothendieck tensor category $\cat A_{\cat P}$ whose kernel when restricted to the catgeory of compact objects $\Mod_{C^*(X;\Fp)}^{c}$ is exactly $\cat P$ (compare \cite[Rem.~5.14]{Balmer20_nilpotence}). Moreover, the family of homological tensor functors
	\[
		\SET{\overbar{h}_{\cat B(\cat P)}  }{\cat P \in \Spc(\Mod_{C^*(X;\Fp)}^{c})}
	\]
	detects tensor-nilpotence. This is of interest since in this example there is no obvious candidate for the construction of residue fields. 
\end{Exa}

\begin{Exa}[Spectral Mackey functors]
	Suppose $\bbE$ is a commutative ring spectrum with the property that $\Spc(\Der(\bbE)^c)$ is noetherian.  It is established in \cite[Thm~15.1]{bhs1} that if $\Der(\bbE)$ is stratified then so is the category of $\bbE$-valued spectral $G$-Mackey functors $\Mack{G}{\bbE}$ for any finite group $G$.  Consequently, there are various categories of spectral Mackey functors to which we can apply \cref{thm:homological}.  See \cite{PatchkoriaSandersWimmer22} and \cite{bhs1} for further discussion of these examples.
\end{Exa}

\bibliographystyle{alphasort}\bibliography{TG-articles}
\end{document}